\documentclass[a4paper, 10pt, notitlepage]{article}

\usepackage{amsthm, amsmath, amssymb, latexsym}
\usepackage{mathrsfs}

\theoremstyle{plain}
\newtheorem{theorem}{Theorem}[section]

\newtheorem{corollary}{Corollary}[section]
\newtheorem{proposition}{Proposition}[section]

\theoremstyle{definition}
\newtheorem{definition}{Definition}[section]

\theoremstyle{remark}
\newtheorem{remark}{Remark}[section]

\DeclareMathOperator*{\argmin}{arg\,min}

\DeclareMathOperator{\dist}{dist}

\DeclareMathOperator{\lineal}{lin}

\DeclareMathOperator{\trace}{Tr}

\author{M.V. Dolgopolik}
\title{A Unified Approach to the Global Exactness of Penalty and Augmented Lagrangian Functions I: Parametric Exactness}

\begin{document}

\maketitle

\begin{abstract}
In this two-part study we develop a unified approach to the analysis of the global exactness of various penalty and
augmented Lagrangian functions for constrained optimization problems in finite dimensional spaces. This approach allows
one to verify in a simple and straightforward manner whether a given penalty/augmented Lagrangian function is exact,
i.e. whether the problem of unconstrained minimization of this function is equivalent (in some sense) to the original
constrained problem, provided the penalty parameter is sufficiently large. Our approach is based on the so-called
localization principle that reduces the study of global exactness to a local analysis of a chosen merit function near
globally optimal solutions. In turn, such local analysis can usually be performed with the use of sufficient optimality
conditions and constraint qualifications.

In the first paper we introduce the concept of global parametric exactness and derive the localization
principle in the parametric form. With the use of this version of the localization principle we recover existing simple
necessary and sufficient conditions for the global exactness of linear penalty functions, and for the existence of
augmented Lagrange multipliers of Rockafellar-Wets' augmented Lagrangian. Also, we obtain completely new necessary and
sufficient conditions for the global exactness of general nonlinear penalty functions, and for the global exactness of a
continuously differentiable penalty function for nonlinear second-order cone programming problems. We briefly
discuss how one can construct a continuously differentiable exact penalty function for nonlinear semidefinite
programming problems, as well.
\end{abstract}

\section{Introduction}

One of the main approaches to the solution of a constrained optimization problem consists in the reduction of this
problem to an unconstrained one (or a sequence of unconstrained problems) with the use of \textit{merit} (or
\textit{auxiliary}) functions. Such merit functions are usually defined as a certain convolution of the objective
function and constraints, and they almost always include the penalty parameter that must be properly chosen for the
reduction to work. This approach led to the development of various penalty and barrier methods 
\cite{FiaccoMcCormick,AuslenderCominettiHaddou,BenTal,Auslender}, primal-dual methods based on the use of augmented
Lagrangians \cite{BirginMartinez_book} and many other methods of constrained optimization. 

There exist numerous results on the duality theory for various merit functions, such as penalty and augmented Lagrangian
functions. A modern general formulation of the augmented Lagrangian duality for nonconvex problems based on a geometric
interpretation of the augmented Lagrangian in terms of subgradients of the optimal value function was proposed by
Rockafellar and Wets in \cite{RockWets}, and further developed in
\cite{HuangYang2003,ZhouYang2004,HuangYang2005,NedichOzdaglar}. Let us also mention several extensions
\cite{GasimovRubinov,BurachikRubinov,ZhouYang2006,ZhangYang2008,ZhouYang2009,BurachikIusemMelo,ZhouYang2012,
WangYangYang2014} of this augmented Lagrangian duality theory aiming at including some other augmented Lagrangian and
penalty functions into the unified framework proposed in \cite{RockWets}. A general duality theory for
\textit{nonlinear} Lagrangian and penalty functions was developed in
\cite{Rubinov2000,RubinovYang2003,PenotRubinov2005,WangYangYang2007}. Another general approach to the study of duality
based on the \textit{image space analysis} was systematically studied in
\cite{EvtushenkoRubinovZhadanI,Giannessi_book,Giannessi2007,Mastroeni2012,LiFengZhang2013,ZhuLi2014,ZhuLi2014_2,
XuLi2014}. 

In contract to duality theory, few attempts
\cite{EvtushenkoZhadan,EvtushenkoZhadan_In_Collection,DiPilloGrippo89,DiPillo1994,EvtushenkoRubinovZhadanII} have been
made to develop a general theory of a \textit{global exactness} of merit functions, despite the abundance of particular
results on the exactness of various penalty/augmented Lagrangian functions. Furthermore, the existing general results on
global exactness are unsatisfactory, since they are very restrictive and cannot be applied to many particular cases.

Recall that a penalty function is called exact iff its points of global minimum coincide with globally optimal solutions
of the constrained optimization problem under consideration. The concept of exactness of a linear penalty function was
introduced by Eremin \cite{Eremin} and Zangwill \cite{Zangwill} in the mid-1960s, and was further investigated by many
researches
(see~\cite{Pietrzykowski,EvansGouldTolle,Bertsekas,HanMangasarian,IoffeNSC,Rosenberg,Mangasarian,Burke,WuBaiYangZhang,
Antczak,Demyanov,Demyanov0,DemyanovDiPilloFacchinei,DiPilloGrippo,DiPilloGrippo89,ExactBarrierFunc,Zaslavski,
Dolgopolik_UT} and the references therein). A class of \textit{continuously differentiable} exact penalty functions was
introduced by Fletcher \cite{Fletcher70} in 1970. Fletcher's penalty functions was modified and thoroughly investigated
in
\cite{Fletcher70,Fletcher73,MukaiPolak,GladPolak,BoggsTolle1980,Bertsekas_book,HanMangasarian_C1PenFunc,DiPilloGrippo85,
DiPilloGrippo86, Lucidi92, ContaldiDiPilloLucidi_93,FukudaSilva,AndreaniFukuda}. Di Pillo and Grippo proposed to
consider an \textit{exact augmented Lagrangian function} \cite{DiPilloGrippo1979} in 1979. This class of augmented
Lagrangian functions was studied and applied to various optimization problems in
\cite{DiPilloGrippo1980,DiPilloGrippo1982,Lucidi1988,DiPilloLucidiPalagi1993,DiPilloLucidi1996,DiPilloLucidi2001,
DiPilloEtAl2002,DiPilloLiuzzi2003,DuZhangGao2006,DuLiangZhang2006,LuoWuLiu2013,DiPilloLiuzzi2011,FukudaLourenco},
while a general theory of globally exact augmented Lagrangian functions was developed by the author in
\cite{Dolgopolik_GSP}. The theory of nonlinear exact penalty functions was developed by Rubinov and his colleagues  
\cite{RubinovGloverYang1999,RubinovYangBagirov2002,RubinovGasimov2003,RubinovYang2003,YangHuang_NonlinearPenalty2003}
in the late 1990s and the early 2000s.  Finally, a new class of exact penalty functions was introduced by Huyer and
Neumaier \cite{HuyerNeumaier} in 2003. Later on, this class of penalty functions was studied by many researchers, and
applied to various optimization problems, including optimal control problems 
\cite{Bingzhuang,WangMaZhou,LiYu,MaLiYiu,LinWuYu,JianLin,LinLoxton,MaZhang2015,ZhengZhang2015,Dolgopolik_OptLet,
Dolgopolik_OptLet2,DolgopolikMV_UT_2}.

It should be noted that the problem of the existence of \textit{global saddle points} of augmented Lagrangian functions
is closely related to the exactness property of these functions. This problem was studied for general cone constrained
optimization problems in \cite{ShapiroSun,ZhouZhouYang2014}, for mathematical programming problems in
\cite{LiuYang,WangLi2009,LiuTangYang2009,LuoMastroeniWu2010,ZhouXiuWang,WuLuo2012b,WangZhouXu2009,SunLi,WangLiuQu}, for
nonlinear second order cone programming problems in \cite{ZhouChen2015}, for nonlinear semidefinite programming problems
in \cite{WuLuoYang2014,LuoWuLiu2015}, and for semi-infinite programming problems in
\cite{RuckmannShapiro,BurachikYangZhou2017}. A general theory of the existence of global saddle point of augmented
Lagrangian functions for cone constrained optimization problems was presented in \cite{Dolgopolik_GSP}. Finally, there
is also a problem of the existence of \textit{augmented Lagrange multipliers}, which can be viewed as the study of the
global exactness of Rockafellar-Wets' augmented Lagrangian function. Various results on the existence of augmented
Lagrange multipliers were obtained in
\cite{ShapiroSun,ZhouZhouYang2014,Dolgopolik_ALRW,RuckmannShapiro,KanSong,KanSong2,BurachikYangZhou2017}.

The anaylsis of the proofs of the main results of the aforementioned papers indicates that the underlying ideas of
these papers largely overlap. Our main goal is to unveil the core idea behind these result, and present a general theory
of the global exactness of penalty and augmented Lagrangian functions for finite dimensional constrained optimization
problems that can be applied to all existing penalty and augmented Lagrangian functions. The central result of our
theory is the so-called \textit{localization principle}. This principle allows one to reduce the study of the global
exactness of a given merit function to a local analysis of the behaviour of this function near globally optimal
solutions of the original constrained problem. In turn, such local analysis can be usually performed with the use of
sufficient optimality conditions and/or constraint qualifications. Thus, the localization principle furnishes one with a
simple technique for proving the global exactness of almost any merit function with the use of the standard tools of
constrained optimization (namely, constraint qualifications and optimality conditions). The localization principle was
first derived by the author for linear penalty functions in \cite{Dolgopolik_UT}, and was further extended to other
penalty and augmented Lagrangian functions in \cite{Dolgopolik_GSP,DolgopolikMV_UT_2,Dolgopolik_ALRW}

In order to include almost all imaginable penalty and augmented Lagrangian functions into the general theory, we
introduce and study the concept of global exactness for an arbitrary function depending on the primal variables, the
penalty parameter and some additional parameters, and do not impose any assumptions on the structure of this function.
Instead, natural assumptions on the behaviour of this function arise within the localization principle as necessary and
sufficient conditions for the global exactness. 

It might seem natural to adopt the approach of the image space analysis
\cite{Giannessi_book,Giannessi2007,Mastroeni2012,LiFengZhang2013,ZhuLi2014,ZhuLi2014_2,XuLi2014} for the study of global
exactness. However, the definition of separation function from the image space analysis imposes some assumptions on the
structure of admissible penalty/aug\-mented Lagrangian functions, which create some unnecessary restrictions. In
contrast, our approach to the global exactness avoids any such assumptions.

Finally, let us note that there are several possible ways to introduce the concept of the global exactness of a merit
function. Each part of this two-part study is devoted to the analysis of one of the possible approaches to the
definition of global exactness. In this paper we study the so-called global \textit{parametric} exactness, which
naturally arises during the study of various exact penalty functions and augmented Lagrange multipliers.

The paper is organized as follows. In Section~\ref{Sect_ParametricExactness} we introduce the definition of global
parametric exactness and derive the localization principle in the parametric form. This version of localization
principle is applied to the study of the global exactness of several penalty and augmented Lagrangian in 
Section~\ref{Sect_Appl_ParametricExactness}. In particular, in this section we recover existing necessary and sufficient
conditions for the global exactness of linear penalty function, and for the existence of augmented Lagrange multipliers.
We also obtain completely new necessary and sufficient conditions for the global exactness of a continuously
differentiable penalty function for nonlinear second-order cone programming problems, and briefly discuss how one can
define a globally exact continuously differentiable penalty function for nonlinear semidefinite programming problems.
Necessary preliminary results are given in Section~\ref{Sect_Preliminaries}.

\section{Preliminaries}
\label{Sect_Preliminaries}

Let $X$ be a finite dimensional normed space, and $M, A \subset X$ be nonempty sets. Throughout this article, we study
the following optimization problem
$$
  \min f(x) \quad \text{subject to} \quad x \in M, \quad x \in A,
  \eqno{(\mathcal{P})}
$$
where $f\colon X \to \mathbb{R} \cup \{ + \infty \}$ is a given function. Denote by $\Omega = M \cap A$ the set of
feasible points of this problem. From this point onwards, we suppose that there exists $x \in \Omega$ such that
$f(x) < + \infty$, and that there exists a globally optimal solution of $(\mathcal{P})$. 

Our aim is to somehow ``get rid'' of the constraint $x \in M$ in the problem $(\mathcal{P})$ with the use of an
auxiliary function $F(\cdot)$. Namely, we want to construct an auxiliary function $F(\cdot)$ such that globally optimal
solutions of the problem $(\mathcal{P})$ can be easily recovered from points of global minimum of $F(\cdot)$ on the
set $A$. To be more precise, our aim is to develop a general theory of such auxiliary functions. 

\begin{remark}
It should be underlined that only the constraint $x \in M$ is incorporated into an auxiliary function $F(\cdot)$, while
the constraint $x \in A$ must be taken into account explicitly. Usually, the set $A$ represents ``simple'' constrains
such as bound or linear ones. Alternatively, one can utilize one auxiliary function in order to ``get rid'' of one kind
of constraints, and then utilize a different type of auxiliary functions in order to ``get rid'' of other kind of
constraints. Overall, the differentiation of the constraints onto the main ones ($x \in M$) and the additional ones 
($x \in A$) gives one more flexibility in the choice of the tools for solving constrained optimization problems.
\end{remark}

Let $\Lambda$ be a nonempty set of parameters that are denoted by $\lambda$, and let $c > 0$ be 
\textit{the penalty parameter}. Hereinafter, we suppose that a function 
$F \colon X \times \Lambda \times (0, + \infty) \to \mathbb{R} \cup \{ + \infty \}$, $F = F(x, \lambda, c)$, is given.
A connection between this function and the problem $(\mathcal{P})$ is specified below. 

The function $F$ can be, for instance, a penalty function with $\Lambda$ being the empty set or an augmented Lagrangian
function with $\lambda$ being a Lagrange multiplier. However, in order not to restrict ourselves to any specific case,
we call $F(x, \lambda, c)$ \textit{a separating function} for the problem $(\mathcal{P})$. 

\begin{remark}
The motivation behind the term ``separating function'' comes from a geometric interpretation of many penalty and
augmented Lagrangian function as nonlinear functions separating some nonconvex sets. This point of view on penalty and
augmented Lagrangian functions is systematically utilized within the image space analysis
\cite{Giannessi_book,Giannessi2007,Mastroeni2012,LiFengZhang2013,LuoWuLiu2013,ZhuLi2014,ZhuLi2014_2,XuLi2014}.
\end{remark}

\begin{remark}
Let us note that since we consider only separating functions depending on the penalty parameter $c > 0$, the so-called
\textit{objective penalty functions} (see, e.g., \cite{EvtushenkoRubinovZhadanII,MengHuDang2004,MengDangJiang2013})
cannot be considered within our theory.
\end{remark}

\section{A General Theory of Parametric Exactness}
\label{Sect_ParametricExactness}

In the first part of our study, we consider the simplest case when one minimizes the function $F(x, \lambda, c)$ with
respect to $x$, and views $\lambda$ as a tuning parameter. Let us introduce the formal definition of exactness of the
function $F(x, \lambda, c)$ in this case.

\begin{definition}
The separating function $F(x, \lambda, c)$ is said to be \textit{globally parametrically exact} iff there exist
$\lambda^* \in \Lambda$ and $c^* > 0$ such that for any $c \ge c^*$ one has
$$
  \argmin_{x \in A} F(x, \lambda^*, c) = \argmin_{x \in \Omega} f(x).
$$
The greatest lower bound of all such $c^* > 0$ is called \textit{the least exact penalty parameter} of the function
$F(x, \lambda^*, c)$, and is denoted by $c^*(\lambda^*)$, while $\lambda^*$ is called \textit{an exact tuning
parameter}.
\end{definition}

Thus, if $F(x, \lambda, c)$ is globally parametrically exact and an exact tuning parameter $\lambda^*$ is known, then
one can choose sufficiently large $c > 0$ and minimize the function $F(\cdot, \lambda^*, c)$ over the set $A$ in order
to find globally optimal solutions of the problem $(\mathcal{P})$. In other words, if the function $F(x, \lambda, c)$ is
globally exact, then one can remove the constraint $x \in M$ with the use of the function $F(x, \lambda, c)$ without
loosing any information about globally optimal solutions of the problem $(\mathcal{P})$.

Our main goal is to demonstrate that the study of the global parametric exactness of 
the separating function $F(x, \lambda, c)$ can be easily reduced to the study of a local behaviour of $F(x, \lambda, c)$
near globally optimal solutions of the problem $(\mathcal{P})$. This reduction procedure is called \textit{the
localization principle}. 

At first, let us describe a desired local behaviour of the function $F(x, \lambda, c)$ near optimal solutions.

\begin{definition}
Let $x^*$ be a locally optimal solution of the problem $(\mathcal{P})$. The separating function $F(x, \lambda, c)$ is
called \textit{locally parametrically exact} at $x^*$ iff there exist $\lambda^* \in \Lambda$, $c^* > 0$ and  a
neighbourhood $U$ of $x^*$ such that for any $c \ge c^*$ one has
$$
  F(x, \lambda^*, c) \ge F(x^*, \lambda^*, c) \quad \forall \quad x \in U \cap A.
$$
The greatest lower bound of all such $c^* > 0$ is called \textit{the least exact penalty parameter} of the function
$F(x, \lambda^*, c)$ at $x^*$, and is denoted by $c^*(x^*, \lambda^*)$, while $\lambda^*$ is called \textit{an exact
tuning parameter} at $x^*$.
\end{definition}

Thus, $F(x, \lambda, c)$ is locally parametrically exact at a point $x^*$ with an exact tuning parameter $\lambda^*$ iff
there exists $c^* > 0$ such that $x^*$ is a local (uniformly with respect to $c \in [c^*, + \infty)$) minimizer of the
function $F(\cdot, \lambda^*, c)$ on the set $A$. Observe also that if the function $F(x, \lambda, c)$ is nondecreasing
in $c$, then $F(x, \lambda, c)$ is locally parametrically exact at $x^*$ with an exact tuning parameter $\lambda^*$ iff
there exists $c^*$ such that $x^*$ is a local minimizer of $F(\cdot, \lambda^*, c^*)$ on the set $A$.

Recall that $c > 0$ in $F(x, \lambda, c)$ is called \textit{the penalty parameter}; however, a connection of 
the parameter $c$ with penalization is unclear from the definition of the function $F(x, \lambda, c)$. We need the
following definition in order to clarify this connection.

\begin{definition}
Let $\lambda^* \in \Lambda$ be fixed. One says that $F(x, \lambda, c)$ is a \textit{penalty-type} separating function
for $\lambda = \lambda^*$ iff there exists $c_0 > 0$ such that if
\begin{enumerate}
\item{$\{ c_n \} \subset [c_0, + \infty)$ is an increasing unbounded sequence;
}

\item{$x_n \in \argmin_{x \in A} F(x, \lambda^*, c_n)$, $n \in \mathbb{N}$;
}

\item{$x^*$ is a cluster point of the sequence $\{ x_n \}$,
}
\end{enumerate}
then $x^*$ is a globally optimal solution of the problem $(\mathcal{P})$.
\end{definition}

Roughly speaking, $F(x, \lambda, c)$ is a penalty-type separating function for $\lambda = \lambda^*$ iff global
minimizers of $F(\cdot, \lambda^*, c)$ on the set $A$ tend to globally optimal solutions of the problem $(\mathcal{P})$
as $c \to + \infty$. Thus, if the separating function $F(x, \lambda, c)$ is of penalty-type, then $c$ plays the role of
penalty parameter, since an increase of $c$ forces global minimizers of $F(\cdot, \lambda^*, c)$ to get closer to
the feasible set of the problem $(\mathcal{P})$.

Note that if the function $F(\cdot, \lambda^*, c)$ does not attain a global minimum on the set $A$ for any $c$ greater
than some $c_0 > 0$, then, formally, $F(x, \lambda, c)$ is a penalty-type separating function for $\lambda = \lambda^*$.
Similarly, if all sequences $\{ x_n \}$, such that $x_n \in \argmin_{x \in A} F(x, \lambda^*, c_n)$, $n \in \mathbb{N}$
and $c_n \to + \infty$ as $n \to \infty$, do not have cluster points, then $F(x, \lambda, c)$ is a penalty-type
separating function for $\lambda = \lambda^*$, as well. Therefore we need an additional definition that allows one to
exclude such pathological behaviour of the function $F(x, \lambda, c)$ as $c \to \infty$ (see~\cite{Dolgopolik_UT},
Sections 3.2--3.4, for the motivation behind this definition). 

Recall that $A$ is a subset of a finite dimensional normed space $X$.

\begin{definition}
Let $\lambda^* \in \Lambda$ be fixed. The separating function $F(x, \lambda, c)$ is said to be \textit{non-degenerate}
for $\lambda = \lambda^*$ iff there exist $c_0 > 0$ and $R > 0$ such that for any $c \ge c_0$ the function 
$F(\cdot, \lambda^*, c)$ attains a global minimum on the set $A$, and there 
exists $x(c) \in \argmin_{x \in A} F(x, \lambda^*, c)$ such that $\| x(c) \| \le R$.
\end{definition}

Roughly speaking, the non-degeneracy condition does not allow global minimizers of $F(\cdot, \lambda^*, c)$ on the set
$A$ to escape to infinity as $c \to \infty$. Note that if the set $A$ is bounded, then $F(x, \lambda, c)$ is
non-degenerate for $\lambda = \lambda^*$ iff the function $F(\cdot, \lambda^*, c)$ attains a global minimum on the set
$A$ for any $c$ large enough.

Now, we are ready to formulate and prove the localization principle. Recall that $\Omega$ is the feasible set of the
problem $(\mathcal{P})$. Denote by $\Omega^*$ the set of globally optimal solutions of this problem.

\begin{theorem}[Localization Principle in the Parametric Form I] \label{Thrm_LocPrin_Param}
Suppose that the validity of the condition
\begin{equation} \label{InclImpliesExactness}
  \Omega^* \cap \argmin_{x \in A} F(x, \lambda^*, c) \ne \emptyset
\end{equation}
for some $\lambda^* \in \Lambda$ and $c > 0$ implies that $F(x, \lambda, c)$ is globally parametrically exact with the
exact tuning parameter $\lambda^*$. Let also $\Omega$ be closed, and $f$ be l.s.c. on $\Omega$. Then the separating
function $F(x, \lambda, c)$ is globally parametrically exact if and only if there exists $\lambda^* \in \Lambda$ such
that
\begin{enumerate}
\item{$F(x, \lambda, c)$ is of penalty-type and non-degenerate for $\lambda = \lambda^*$;
}

\item{$F(x, \lambda, c)$ is locally parametrically exact with the exact tuning parameter $\lambda^*$ at every globally
optimal solution of the problem $(\mathcal{P})$.
}
\end{enumerate}
\end{theorem}

\begin{proof}
Suppose that $F(x, \lambda, c)$ is globally parametrically exact with an exact tuning parameter $\lambda^*$. Then  
for any $c > c^*(\lambda^*)$ one has
$$
  \argmin_{x \in A} F(x, \lambda^*, c) = \Omega^*.
$$
In other words, for any $c > c^*(\lambda^*)$ every globally optimal solution $x^*$ of the problem $(\mathcal{P})$ is a
global (and hence local uniformly with respect to $c \in (c^*(\lambda^*), + \infty)$) minimizer of 
$F(\cdot, \lambda^*, c)$ on the set $A$. Thus, $F(x, \lambda, c)$ is locally parametrically exact with the exact tuning
parameter $\lambda^*$ at every globally optimal solution of the problem $(\mathcal{P})$.

Fix arbitrary $x^* \in \Omega^*$. Then for any $c > c(\lambda^*)$ the point $x^*$ is a global minimizer of 
$F(\cdot, \lambda^*, c)$, which implies that $F(x, \lambda, c)$ is non-degenerate for $\lambda = \lambda^*$.
Furthermore, if a sequence $\{ x_n \} \subset A$ is such that $x_n \in \argmin_{x \in A} F(x, \lambda^*, c_n)$ for all
$n \in N$, where $c_n \to + \infty$ as $n \to \infty$, then due to the global exactness of $F$ one has that for all $n$
large enough the point $x_n$ coincides with one of the globally optimal solution of $(\mathcal{P})$, which implies that 
$x_n \in \Omega$, and $f(x_n) = \min_{x \in \Omega} f(x)$. Hence applying the facts that $\Omega$ is closed and $f$ is
l.s.c. on $\Omega$ one can easily verify that a cluster point of the sequence $\{ x_n \}$, if exists, is a globally
optimal solution of $(\mathcal{P})$. Thus, $F(x, \lambda, c)$ is a penalty-type separating function 
for $\lambda = \lambda^*$.

Let us prove the converse statement. Our aim is to verify that there exist $c > 0$ and $x^* \in \Omega^*$ such that
\begin{equation} \label{EqualityImplExactness}
  \inf_{x \in A} F(x, \lambda^*, c) = F(x^*, \lambda^*, c). 
\end{equation}
Then taking into account condition \eqref{InclImpliesExactness} one obtains that the separating function 
$F(x, \lambda, c)$ is globally parametrically exact. Arguing by reductio ad absurdum, suppose that
\eqref{EqualityImplExactness} is not valid. Then, in particular, for any $n \in \mathbb{N}$ one has
\begin{equation} \label{ParamExactAdAbsurdum}
  \inf_{x \in A} F(x, \lambda^*, n) < F(x^*, \lambda^*, n) \quad \forall x^* \in \Omega^*.
\end{equation}
By condition~1, the function $F(x, \lambda, c)$ is non-degenerate for $\lambda = \lambda^*$. Therefore there exist
$n_0 \in \mathbb{N}$ and $R > 0$ such that for any $n \ge n_0$ there 
exists $x_n \in \argmin_{x \in A} F(x, \lambda^*, n)$ with $\| x_n \| \le R$. 

Recall that $X$ is a finite dimensional normed space. Therefore there exists a subsequence $\{ x_{n_k} \}$ converging to
some $x^*$. Consequently, applying the fact that $F(x, \lambda, c)$ is a penalty-type separating function 
for $\lambda = \lambda^*$ one obtains that $x^* \in \Omega^*$. By condition~2, $F(x, \lambda, c)$ is locally
parametrically exact at $x^*$ with the exact tuning parameter $\lambda^*$. Therefore there exist $c_0 > 0$ and a
neighbourhood $U$ of $x^*$ such that for any $c \ge c_0$ one has 
\begin{equation} \label{LocalParamExactness}
  F(x, \lambda^*, c) \ge F(x^*, \lambda^*, c) \quad \forall x \in U \cap A.
\end{equation}
Since the subsequence $\{ x_{n_k} \}$ converges to $x^*$, there exists $k_0$ such that for any $k \ge k_0$ one has
$x_{n_k} \in U$. Moreover, one can suppose that $n_k \ge c_0$ for all $k \ge k_0$. Hence with the use of
\eqref{LocalParamExactness} one obtains that
$$
  F(x_{n_k}, \lambda^*, n_k) \ge F(x^*, \lambda^*, n_k),
$$
which contradicts \eqref{ParamExactAdAbsurdum} and the fact that $x_{n_k} \in \argmin_{x \in A} F(x, \lambda^*, n_k)$.
Thus, $F(x, \lambda, c)$ is globally parametrically exact.
\end{proof}

\begin{remark} \label{Rmrk_LocPrincipleParamForm}
{(i)~Condition \eqref{InclImpliesExactness} simply means that in order to prove the global
parametric exactness of $F(x, \lambda, c)$ it is sufficient to check that at least one globally optimal solution of 
the problem $(\mathcal{P})$ is a point of global minimum of the function $F(\cdot, \lambda^*, c)$ instead of verifying
that the sets $\Omega^*$ and $\argmin_{x \in A} F(x, \lambda^*, c)$ actually coincide. It should be pointed out that in
most particular cases the validity of condition \eqref{InclImpliesExactness} is equivalent to global parametric
exactness. In fact, the equivalence between \eqref{InclImpliesExactness} and global parametric exactness automatically,
i.e. without any additional assumptions, holds true in all but one example (see subsection~\ref{Example_ALRW} below)
presented in this article. Note, finally, that condition \eqref{InclImpliesExactness} is needed only to prove the ``if''
part of the theorem.
}

\noindent{(ii)~The theorem above describes how to construct a globally exact separating function $F(x, \lambda, c)$.
Namely, one has to ensure that a chosen function $F(x, \lambda, c)$ is of penalty-type (which can be guaranteed by
adding a penalty term to the function $F(x, \lambda, c)$), non-degenerate (which can usually be guaranteed by the
introduction of a barrier term into the function $F(x, \lambda, c)$) and is locally exact near all globally optimal
solutions of the problem $(\mathcal{P})$, which is typically done with the use of constraint qualifications (metric
(sub-)regularity assumptions) and/or sufficient optimality conditions. Below, we present several particular examples
illustrating the usage of the previous theorem.
}

\noindent{(iii)~Note that the previous theorem can be reformulated as a theorem describing necessary and sufficient
conditions for a tuning parameter $\lambda^* \in \Lambda$ to be exact. It should also be mentioned that the theorem
above can be utilized in order to obtain necessary and/or sufficient conditions for the uniqueness of an exact tuning
parameter, In particular, it is easy to see that a globally exact tuning parameter $\lambda^*$ is unique, if there
exists $x^* \in \Omega^*$ such that a locally exact tuning parameter at $x^*$ is unique.
}
\end{remark}

The theorem above can be vaguely formulated as follows. The separating function $F(x, \lambda, c)$ is globally
parametrically exact iff it is of penalty-type, non-degenerate and locally exact at every globally optimal solution
of the problem $(\mathcal{P})$. Thus, under natural assumptions the function $F(x, \lambda, c)$ is globally exact iff
it is exact near globally optimal solutions of the original problem. That is why Theorem~\ref{Thrm_LocPrin_Param} is
called \textit{the localization principle}.

Let us reformulate the localization principle in the form that is slightly more convenient for applications.

\begin{theorem}[Localization Principle in the Parametric Form II] \label{Thrm_LocPrin_Param_SublevelSets}
Suppose that the validity of the condition
$$
  \Omega^* \cap \argmin_{x \in A} F(x, \lambda^*, c) \ne \emptyset
$$
for some $\lambda^* \in \Lambda$ and $c > 0$ implies that $F(x, \lambda, c)$ is globally parametrically exact with the
exact tuning parameter $\lambda^*$.  Let also the sets $A$ and $\Omega$ be closed, the objective function $f$ be l.s.c.
on $\Omega$, and the function $F(\cdot, \lambda, c)$ be l.s.c. on $A$ for all $\lambda \in \Lambda$ and $c > 0$. Then
the separating function $F(x, \lambda, c)$ is globally parametrically exact if and only if there 
exists $\lambda^* \in \Lambda$ such that
\begin{enumerate}
\item{$F(x, \lambda, c)$ is of penalty-type for $\lambda = \lambda^*$;
}

\item{there exist $c_0 > 0$, $x^* \in \Omega^*$ and a bounded set $K \subset A$ such that
\begin{equation} \label{SubleveBoundedness_Param}
  S(c, x^*) := \Big\{ x \in A \mid F(x, \lambda^*, c) < F(x^*, \lambda^*, c) \Big\} \subset K \quad \forall c \ge c_0;
\end{equation}
}

\item{$F(x, \lambda, c)$ is locally parametrically exact at every globally optimal solution of the problem
$(\mathcal{P})$ with the exact tuning parameter $\lambda^*$.
}
\end{enumerate}
\end{theorem}

\begin{proof}
Suppose that $F(x, \lambda, c)$ is globally parametrically exact with the exact tuning parameter $\lambda^*$. Then, as
it was proved in Theorem~\ref{Thrm_LocPrin_Param}, $F(x, \lambda, c)$ is a penalty-type separating function 
for $\lambda = \lambda^*$, and $F(x, \lambda, c)$ is locally parametrically exact with the exact tuning parameter
$\lambda^*$ at every globally optimal solution of the problem $(\mathcal{P})$. Furthermore, from the definition of
global exactness it follows that $S(c, x^*) = \emptyset$ for all $c > c^*(\lambda^*)$ and $x^* \in \Omega^*$, which
implies that \eqref{SubleveBoundedness_Param} is satisfied for all $c_0 > c^*(\lambda^*)$, $x^* \in \Omega^*$ and any
bounded set $K$.

Let us prove the converse statement. By our assumption there exist $c_0 > 0$ and $x^* \in \Omega^*$ such that for all
$c \ge c_0$ the sublevel set $S(c, x^*)$ is contained in a bounded set $K$ and, thus, is bounded. Therefore taking into
account the facts that the function $F(\cdot, \lambda^*, c)$ is l.s.c. on $A$, and the set $A$ is closed one obtains
that $F(\cdot, \lambda^*, c)$ attains a global minimum on the set $A$ at a point $x(c) \in K$ 
(if $S(c, x^*) = \emptyset$ for some $c \ge c_0$, then $x(c) = x^*$). From the fact that $K$ is bounded it follows that
that there exists $R > 0$ such that $\| x(c) \| \le R$ for all $c \ge c_0$, which implies that $F(x, \lambda, c)$ is
non-degenerate for $\lambda = \lambda^*$. Consequently, applying Theorem~\ref{Thrm_LocPrin_Param} one obtains the
desired result.
\end{proof}

Note that the definition of global parametric exactness does not specify how the optimal value of the problem
$(\mathcal{P})$ and the infimum of the function $F(\cdot, \lambda^*, c)$ over the set $A$ are connected. In some
particular cases (see subsection~\ref{Example_ALRW} below), this fact might significantly complicate the application of
the localization principle. Therefore, let us show how one can incorporate the assumption on the value of
$\inf_{x \in A} F(x, \lambda^*, c)$ into the localization principle.

\begin{definition}
The separating function $F(x, \lambda, c)$ is said to be \textit{strictly globally parametrically exact} if 
$F(x, \lambda, c)$ is globally parametrically exact, and there exists $c_0 > 0$ such that
\begin{equation} \label{StrictParamExactness}
  \inf_{x \in A} F(x, \lambda^*, c) = f^* \quad \forall c \ge c_0,
\end{equation}
where $\lambda^*$ is an exact tuning parameter, and $f^* = \inf_{x \in \Omega} f(x)$ is the optimal value of 
the problem $(\mathcal{P})$. An exact tuning parameter satisfying \eqref{StrictParamExactness} is called
\textit{strictly exact}.
\end{definition}

Arguing in a similar way to the proofs of Theorems~\ref{Thrm_LocPrin_Param} and \ref{Thrm_LocPrin_Param_SublevelSets}
one can easily extend the localization principle to the case of strict exactness.

\begin{theorem}[Strengthened Localization Principle in the Parametric Form I] \label{Thrm_LocPrin_Param_Strict}
Suppose that the validity of the conditions
\begin{equation} \label{InclImpliesStrictExactness}
  \Omega^* \cap \argmin_{x \in A} F(x, \lambda^*, c) \ne \emptyset, \quad \min_{x \in A} F(x, \lambda^*, c) = f^*
\end{equation}
for some $\lambda^* \in \Lambda$ and $c > 0$ implies that $F(x, \lambda, c)$ is strictly globally parametrically exact
with $\lambda^*$ being a strictly exact tuning parameter. Let also $\Omega$ be closed, and $f$ be l.s.c. on $\Omega$.
Then the separating function $F(x, \lambda, c)$ is strictly globally parametrically exact if and only if there exists 
$\lambda^* \in \Lambda$ such that 
\begin{enumerate}
\item{$F(x, \lambda, c)$ is of penalty-type and non-degenerate for $\lambda = \lambda^*$;}

\item{$F(x, \lambda, c)$ is locally parametrically exact at every globally optimal solution of the problem
$(\mathcal{P})$ with the exact tuning parameter $\lambda^*$;
}

\item{there exists $c_0 > 0$ such that $F(x^*, \lambda^*, c) = f^*$ for all $x^* \in \Omega^*$ and $c \ge c_0$.}
\end{enumerate}
\end{theorem}

\begin{theorem}[Strengthened Localization Principle in the Parametric Form II]
\label{Thrm_LocPrin_Param_Strict_Sublevel}
Suppose that the validity of the conditions
$$
  \Omega^* \cap \argmin_{x \in A} F(x, \lambda^*, c) \ne \emptyset, \quad \min_{x \in A} F(x, \lambda^*, c) = f^*
$$
for some $\lambda^* \in \Lambda$ and $c > 0$ implies that $F(x, \lambda, c)$ is strictly globally parametrically exact
with $\lambda^*$ being a strictly exact tuning parameter. Let also the sets $A$ and $\Omega$ be closed, the objective
function $f$ be l.s.c. on $\Omega$, and the function $F(\cdot, \lambda, c)$ be l.s.c. on $A$ for 
all $\lambda \in \Lambda$ and $c > 0$. Then the separating function $F(x, \lambda, c)$ is strictly globally
parametrically exact if and only if there exist $\lambda^* \in \Lambda$ and $c_0 > 0$ such that 
\begin{enumerate}
\item{$F(x, \lambda, c)$ is of penalty-type for $\lambda = \lambda^*$;}

\item{there exists a bounded set $K$ such that
$$
  \Big\{ x \in A \Bigm| F(x, \lambda^*, c) < f^* \Big\} \subset K \quad \forall c \ge c_0;
$$
}

\item{$F(x, \lambda, c)$ is locally parametrically exact with the exact tuning parameter $\lambda^*$ at every globally
optimal solution of the problem $(\mathcal{P})$;
}

\item{$F(x^*, \lambda^*, c) = f^*$ for all $x^* \in \Omega^*$ and $c \ge c_0$.}
\end{enumerate}
\end{theorem}

\section{Applications of the Localization Principle}
\label{Sect_Appl_ParametricExactness}

Below, we provide several examples demonstrating how one can apply the localization principle in the parametric form
to the study of the global exactness of various penalty and augmented Lagrangian functions.

\subsection{Example I: Linear Penalty Functions}

We start with the simplest case when the function $F(x, \lambda, c)$ is affine with respect to the penalty parameter $c$
and does not depend on any additional parameters. Let a function $\varphi \colon X \to [0, +\infty]$ be such that
$\varphi(x) = 0$ iff $x \in M$. Define
$$
  F(x, c) = f(x) + c \varphi(x).
$$
The function $F(x, c)$ is called \textit{a linear penalty function} for the problem $(\mathcal{P})$. 

\begin{remark}
In order to rigorously include linear penalty functions (as well as nonlinear penalty functions from the
following two examples) into the theory of parametrically exact separating functions one has to define $\Lambda$ to be a
one-point set, say $\Lambda = \{ - 1 \}$, introduce a new separating function $\widehat{F}(x, -1, c) \equiv F(x, c)$,
and consider the separating function $\widehat{F}(x, \lambda, c)$ instead of the penalty function $F(x, c)$. However,
since this transformation is purely formal, we omit it for the sake of shortness. Moreover, since in the case of penalty
functions the parameter $\lambda$ is absent, it is natural to omit the term ``parametric'', and say that $F(x, c)$ is
globally/locally exact.
\end{remark}

Let us obtain two simple characterizations of the global exactness of the linear penalty function $F(x, c)$ with the use
of the localization principle (Theorems~\ref{Thrm_LocPrin_Param} and \ref{Thrm_LocPrin_Param_SublevelSets}). These
characterizations were first obtained by the author in (\cite{Dolgopolik_UT}, Therems~3.10 and 3.17).

Before we formulate the main result, let us note that $F(x^*, c) = f^*$ for any globally optimal solution $x^*$ of the
problem $(\mathcal{P})$ and for all $c > 0$. Therefore, in particular, the linear penalty function $F(x, c)$ is
globally parametrically exact iff it is strictly globally parametrically exact.

\begin{theorem}[Localization Principle for Linear Penalty Functions]
Let $A$ and $\Omega$ be closed, and let $f$ and $\varphi$ be l.s.c. on $A$. Then the linear penalty function $F(x, c)$
is globally exact if and only if $F(x, c)$ is locally exact at every globally optimal solution of the problem
$(\mathcal{P})$ and one of the following two conditions is satisfied
\begin{enumerate}
\item{$F$ is non-degenerate;
}

\item{there exists $c_0 > 0$ such that the set $\{ x \in A \mid F(x, c_0) < f^* \}$ is bounded.
}
\end{enumerate}
\end{theorem}

\begin{proof}
Note that $F(x^*, c) = f(x^*) = f^*$ for any $x^* \in \Omega^*$ and $c > 0$. Therefore
$$
  \Omega^* \cap \argmin_{x \in A} F(x, c) \ne \emptyset \implies \Omega^* \subset \argmin_{x \in A} F(x, c).
$$
Note also that if $x \notin M$, then either $F(x, c)$ is strictly increasing in $c$ or $F(x, c) = + \infty$ for all 
$c > 0$. On the other hand, if $x \in M$, then $F(x, c) = f(x)$. Consequently, if for some $c_0 > 0$ one has
$\Omega^* \subset \argmin_{x \in A} F(x, c)$, then for any $c > c_0$ one has $\Omega^* = \argmin_{x \in A} F(x, c)$.
Thus, the validity of the condition $\Omega^* \cap \argmin_{x \in A} F(x, c) \ne \emptyset$ for some $c > 0$ implies 
the global exactness of $F(x, c)$.

Our aim, now, is to verify that $F$ is a penalty-type separating function. Then applying
Theorems~\ref{Thrm_LocPrin_Param} and \ref{Thrm_LocPrin_Param_SublevelSets} one obtains the desired result.

Indeed, let $\{ c_n \} \subset (0, + \infty)$ be an increasing unbounded sequence, $x_n \in \argmin_{x \in A} F(x, c)$
for all $n \in \mathbb{N}$, and let $x^*$ be a cluster point of the sequence $\{ x_n \}$. By \cite{Dolgopolik_UT},
Proposition~3.5, one has $\varphi(x_n) \to 0$ as $n \to \infty$. Hence taking into account the facts that $A$ is closed
and $\varphi$ is l.s.c. on $A$ one gets that $x^* \in A$ and $\varphi(x^*) = 0$. Therefore $x^*$ is a feasible point of
the problem $(\mathcal{P})$.

As it was noted above, for any $y^* \in \Omega^*$ one has $F(y^*, c) = f(y^*)$ for all $c > 0$. Hence taking into
account the definition of $x_n$ and the fact that the function $\varphi$ is non-negative one gets 
that $f(x_n) \le f(y^*)$ for all $n \in \mathbb{N}$. Consequently, with the use of the lower semicontinuity of $f$ one
obtains that $f(x^*) \le f(y^*)$, which implies that $x^*$ is a globally optimal solution of the problem
$(\mathcal{P})$. Thus, $F(x, c)$ is a penalty-type separating function.
\end{proof}

Let us also give a different formulation of the localization principle for linear penalty functions in which the
non-degeneracy condition is replaced by some more widely used conditions.

\begin{corollary}
Let $A$ and $\Omega$ be closed, and let $f$ and $\varphi$ be l.s.c. on $A$. Suppose also that one of the following
conditions is satisfied:
\begin{enumerate}
\item{the set $\{ x \in A \mid f(x) < f^* \}$ is bounded;
}

\item{there exist $c_0 > 0$ and $\delta > 0$ such that the function $F(\cdot, c_0)$ is bounded from below on $A$ and
the set $\{ x \in A \mid f(x) < f^*, \: \varphi(x) < \delta \}$ is bounded;
}

\item{there exist $c_0 > 0$ and a feasible point $x_0$ of the problem $(\mathcal{P})$ such that the set 
$\{ x \in A \mid F(x, c_0) \le f(x_0) \}$ is bounded;
}

\item{the function $f$ is coercive on the set $A$, i.e. $f(x_n) \to + \infty$ as $n \to \infty$ for any sequence 
$\{ x_n \} \subset A$ such that $\| x_n \| \to + \infty$ as $n \to \infty$;
}

\item{there exists $c_0 > 0$ such that the function $F(\cdot, c_0)$ is coercive on the set $A$;
}

\item{the function $\varphi$ is coercive on the set $A$ and there exists $c_0 > 0$ such that the function
$F(\cdot, c_0)$ is bounded from below on $A$.
}
\end{enumerate}
Then $F(x, c)$ is globally exact if and only if it is locally exact at every globally optimal solution of the problem
$(\mathcal{P})$.
\end{corollary}

\begin{proof}
One can easily verify that if one of the above assumptions holds true, then the set $\{ x \in A \mid F(x, c_0) < f^* \}$
is bounded for some $c_0 > 0$. Then applying the localization principle for linear penalty functions one obtains the
desired result.
\end{proof}

\begin{remark}
The corollary above provides an example of how  one can reformulate the localization principle in a particular case
with the use of some well-known and widely used conditions such as coercivity or the boundedness of a certain sublevel
set. For the sake of shortness, we do not provide such reformulations of the localization principle for particular
separating function $F(x, \lambda, c)$ studied below. However, let us underline that one can easily reformulate the
localization principle with the use of coercivity-type assumptions in any particular case.
\end{remark}

For the sake of completeness, let us also formulate simple sufficient conditions for the local exactness of the
function $F$. These conditions are well-known (see, e.g. \cite{Dolgopolik_UT}, Theorem~2.4 and Proposition~2.7) and
rely on an error bound for the penalty term $\varphi$.

\begin{proposition}
Let $x^*$ be a locally optimal solution of the problem $(\mathcal{P})$, and $f$ be H\"{o}lder continuous with exponent
$\alpha \in (0, 1]$ in a neighbourhood of $x^*$. Suppose also that there exist $\tau > 0$ and $r > 0$ such that
$$
  \varphi(x) \ge \tau \big[ \dist(x, \Omega) \big]^{\alpha} \quad \forall x \in A \colon \| x - x^* \| < r,
$$
where $\dist(x, \Omega) = \inf_{y \in \Omega} \| x - y \|$. Then the linear penalty function $F(x, c)$ is locally exact 
at $x^*$.
\end{proposition}

\subsection{Example II: Nonlinear Penalty Functions}

Let a function $\varphi \colon X \to [0, +\infty]$ be as above. For the sake of convenience, suppose that the objective
function $f$ is non-negative on $X$. From the theoretical point of view this assumption is not restrictive, since one
can always replace the function $f$ with the function $e^{f(\cdot)}$. Furthermore, it should be noted that the
non-negativity assumption on the objective function $f$ is standard in the theory of nonlinear penalty functions 
(cf.~\cite{RubinovGloverYang1999,RubinovYangBagirov2002,RubinovGasimov2003,RubinovYang2003,
YangHuang_NonlinearPenalty2003}).

Let a function $Q \colon [0, + \infty]^2 \to (- \infty, + \infty]$ be fixed. Suppose that the restriction of $Q$ to the
set $[0, + \infty)^2$ is strictly monotone, i.e. $Q(t_1, s_1) < Q(t_2, s_2)$ for any 
$(t_1, s_1), (t_2, s_2) \in [0, + \infty)^2$ such that $t_1 \le t_2$, $s_1 \le s_2$ and $(t_1, s_1) \ne (t_2, s_2)$.
Suppose also that $Q(+ \infty, s) = Q(t, + \infty) = + \infty$ for all $t, s \in [0, + \infty]$.

Define
$$
  F(x, c) = Q\big( f(x), c \varphi(x) \big).
$$
Then $F(x, c)$ is \textit{a nonlinear penalty function} for the problem $(\mathcal{P})$. This type of nonlinear penalty
functions was studied in
\cite{RubinovGloverYang1999,RubinovYangBagirov2002,RubinovGasimov2003,RubinovYang2003,YangHuang_NonlinearPenalty2003}. 

The simplest particular example of nonlinear penalty function is the function $F(x, c)$ of the form
\begin{equation} \label{NonlinearPenFunc}
  F(x, c) = \Big( \big( f(x) \big)^q + \big( c \varphi(x) \big)^q \Big)^{\frac1q}
\end{equation}
with $q > 0$. Here 
$$
  Q(t, s) = \Big( t^q + s^q \Big)^{\frac1q}.
$$
Clearly, this function is monotone. In this article, the function \eqref{NonlinearPenFunc} is called \textit{the}
$q$-\textit{th order nonlinear penalty function} for the problem $(\mathcal{P})$. Let us note that the least exact
penalty parameter of the $q$-th order nonlinear penalty function is often smaller than the least exact penalty parameter
of the linear penalty function $f(x) + c \varphi(x)$ (see~\cite{RubinovYangBagirov2002,RubinovYang2003} for more
details).

Let us obtain a \textit{new} simple characterization of global exactness of the nonlinear penalty function $F(x, c)$,
which does not rely on any assumptions on the perturbation function for the problem $(\mathcal{P})$
(cf.~\cite{RubinovYangBagirov2002,RubinovYang2003}). Furthermore, to the best of author's knowledge, \textit{exact}
nonlinear penalty functions has only been considered for mathematical programming problems, while our results are
applicable in the general case.

\begin{theorem}[Localization Principle for Nonlinear Penalty Functions]
Let the set $A$ be closed, and the functions $f$, $\varphi$ and $F(\cdot, c)$ be l.s.c. on the set $A$. Suppose also
that $Q(0, s) \to + \infty$ as $s \to +\infty$. Then the nonlinear penalty function $F(x, c)$ is globally exact if and
only if it is locally exact at every globally optimal solution of the problem $(\mathcal{P})$ and one of the two
following assumptions is satisfied:
\begin{enumerate}
\item{the function $F(x, c)$ is non-degenerate;
}

\item{there exists $c_0 > 0$ such that the set $\{ x \in A \mid Q(f(x), c_0 \varphi(x)) < Q(f^*, 0) \}$ is bounded.
}
\end{enumerate}
\end{theorem}

\begin{proof}
From the fact that $Q$ is strictly monotone it follows that for any $c > 0$ one has
$$
  F(x, c) = Q(f(x), c \varphi(x)) = Q(f(x), 0) > Q(f^*, 0) \quad \forall x \in \Omega \setminus \Omega^*.
$$
Furthermore, if for some $c_0 > 0$ one has
$$
  \inf_{x \in A} F(x, c_0) := \inf_{x \in A} Q(f(x), c_0 \varphi(x)) = Q(f^*, 0),
$$
then applying the strict mononicity of $Q$ again one obtains that for any $c > c_0$ the following inequality holds true
$$
  F(x, c) = Q(f(x), c \varphi(x)) > Q(f^*, 0) \quad \forall x \in A \setminus \Omega.
$$
Therefore the validity of the condition $\Omega^* \cap \argmin_{x \in A} F(x, c_0) \ne \emptyset$ for some $c_0 > 0$ is
equivalent to the global exactness of $F(x, c)$ by virtue of the fact that for any $c > 0$ and 
$x^* \in \Omega^*$ one has $F(x^*, c) = Q(f^*, 0)$.

Let us verify that $F$ is a penalty-type separating function. Then applying Theorems~\ref{Thrm_LocPrin_Param} and
\ref{Thrm_LocPrin_Param_SublevelSets} one obtains the desired result.

Indeed, let $\{ c_n \} \subset (0, + \infty)$ be an increasing unbounded sequence, 
$x_n \in \argmin_{x \in A} F(x, c_n)$ for all $n \in \mathbb{N}$, and let $x^*$ be a cluster point of 
the sequence $\{ x_n \}$. Let us check, at first, that $\varphi(x_n) \to 0$ as $n \to \infty$. Arguing by reductio ad
absurdum, suppose that there exist $\varepsilon > 0$ and a subsequence $\{ x_{n_k} \}$ of the sequence $\{ x_n \}$
such that $\varphi(x_{n_k}) > \varepsilon$ for all $k \in \mathbb{N}$. Hence applying the monotonicity of $Q$ one
obtains that
$$
  F(x_{n_k}, c_{n_k}) := Q\big( f(x_{n_k}), c_{n_k} \varphi(x_{n_k}) \big) \ge Q(0, c_{n_k} \varepsilon) 
  \quad \forall k \in \mathbb{N}.
$$
Consequently, taking into account the fact that $Q(0, s) \to + \infty$ as $s \to +\infty$ one gets that
$F(x_{n_k}, c_{n_k}) \to + \infty$ as $k \to \infty$, which contradicts the fact that 
\begin{equation} \label{NonlinPenFunc_InfVsOptValue}
  \inf_{x \in A} F(x, c) \le F(y^*, c) = Q(f^*, 0) < + \infty \quad \forall c > 0, \: y^* \in \Omega^*
\end{equation}
(the inequality $Q(f^*, 0) < + \infty$ follows from the strict monotonicity of $Q$). Thus, $\varphi(x_n) \to 0$ as $n
\to \infty$. Applying the fact that $A$ is closed and $\varphi$ is l.s.c. on $A$ one gets that the cluster point $x^*$
is a feasible point of the problem $(\mathcal{P})$.

Note that from \eqref{NonlinPenFunc_InfVsOptValue} and the monotonicity of $Q$ it follows that
$f(x_n) \le f^*$ for all $n \in \mathbb{N}$. Hence taking into account the fact that $f$ is l.s.c. on $A$ one obtains
that $f(x^*) \le f^*$, which implies that $x^*$ is a globally optimal solution of $(\mathcal{P})$. Thus, $F(x, c)$ is a
penalty-type separating function.
\end{proof}

Let us also obtain \textit{new} simple sufficient conditions for the local exactness of the function $F(x, c)$ which
can be applied to the $q$-th order nonlinear penalty function with $q \in (0, 1)$. Note that since the function $Q$ is
strictly monotone, the point $(0, 0)$ is a global minimizer of $Q$ on the set $[0, + \infty] \times [0, + \infty]$.
Therefore, if $x^*$ is a locally optimal solution of $(\mathcal{P})$ such that $f(x^*) = 0$, then $x^*$ is a global
minimizer of $F(\cdot, c)$ on $A$ for all $c > 0$, which implies that $F(x, c)$ is locally exact at $x^*$. Thus, it is
sufficient to consider the case $f(x^*) > 0$.

\begin{theorem}
Let $x^*$ be a locally optimal solution of the problem $(\mathcal{P})$ such that $f(x^*) > 0$. Suppose that $f$ is
H\"{o}lder continuous with exponent $\alpha \in (0, 1]$ near $x^*$ , and there exist $\tau > 0$ and $r > 0$ such that
\begin{equation} \label{NonlinPenFunc_MetricSubreg}
  \varphi(x) \ge \tau [\dist(x, \Omega)]^{\alpha} \quad \forall x \in A \colon \| x - x^* \| < r.
\end{equation}
Suppose also that there exist $t_0 > 0$ and $c_0 > 0$ such that
\begin{equation} \label{NonlinearPenFunc_ConvFuncAssump}
  Q\big( f(x^*) - t, c_0 t \big) \ge Q(f(x^*), 0) \quad \forall t \in [0, t_0).
\end{equation}
Then the nonlinear penalty function $F(x, c)$ is locally exact at $x^*$.
\end{theorem}

\begin{proof}
Since $f$ is H\"{o}lder continuous with exponent $\alpha$ near the locally optimal solution $x^*$ of the problem
$(\mathcal{P})$, there exist $L > 0$ and $\delta < r$ such that
$$
  f(x) \ge f(x^*) - L \big[\dist(x, \Omega)\big]^{\alpha} \ge 0 \quad \forall x \in A \colon \| x - x^* \| < \delta
$$
(\cite{Dolgopolik_UT}, Proposition~2.7). Consequently, applying \eqref{NonlinPenFunc_MetricSubreg} and the fact that
$Q$ is monotone one obtains that for any $x \in A$ with $\| x - x^* \| < \delta$ one has
$$
  Q\big( f(x), c \varphi(x) \big) \ge 
  Q\Big( f(x^*) - L \big[\dist(x, \Omega)\big]^{\alpha}, c \tau \big[\dist(x, \Omega)\big]^{\alpha} \Big).
$$
Hence with the use of \eqref{NonlinearPenFunc_ConvFuncAssump} one gets that there exists $t_0 > 0$ and $c_0 > 0$ such
that for any $c \ge L c_0 / \tau$ one has
$$
  Q\big( f(x), c \varphi(x) \big) \ge Q( f(x^*), 0) 
  \quad \forall x \in A \colon \| x - x^* \| < 
  \min\left\{ \delta, \left( \frac{t_0}{L} \right)^{1 / \alpha} \right\},
$$
which implies that $F(x, c)$ is locally exact at $x^*$ and $c^*(x^*) \le c_0 / \tau$.
\end{proof}

\begin{remark}
Assumption~\eqref{NonlinearPenFunc_ConvFuncAssump} always holds true for the $q$-th order nonlinear penalty function
with $0 < q \le 1$. Indeed, applying the fact that the function $\omega(t) = t^q$ is H\"{o}lder continuous with exponent
$q$ and the H\"{o}lder coefficient $C = 1$ on $[0, + \infty)$ one obtains that
$$
  \big( f(x^*) - t \big)^q + c^q t^q \ge f(x^*)^q - t^q + c^q t^q \ge f(x^*)^q
$$
for any $t \in [0, f(x_*))$ and $c \ge 1$. Hence 
$$
  Q( f(x^*) - t, c t ) \ge Q( f(x^*), 0 ) \quad \forall t \in [0, f(x_*)) \quad \forall c \ge 1,
$$
which implies the required result.
\end{remark}

\begin{remark}
Note that assumption~\eqref{NonlinearPenFunc_ConvFuncAssump} in the theorem above cannot be strengthened. Namely, one
can easily verify that if the nonlinear penalty function $F$ is locally exact at a locally optimal solution $x^*$ for
all Lipschitz continuous functions $f$ and all function $\varphi$ satisfying \eqref{NonlinPenFunc_MetricSubreg}, then
\eqref{NonlinearPenFunc_ConvFuncAssump} holds true (one simply has to define $f(x) = - L \dist(x, \Omega)$ and 
$\varphi(x) = \dist(x, \Omega)$). Note also that the $q$-th order nonlinear penalty function does not satisfy 
assumption~\eqref{NonlinearPenFunc_ConvFuncAssump} for $q > 1$.
\end{remark}

\subsection{Example III: Continuously Differentiable Exact Penalty Functions}

In this section, we utilize the localization principle in order to improve existing results on the global exactness of 
continuously differentiable exact penalty functions. A continuously differentiable exact penalty function for
mathematical programming problems was introduced by Fletcher in \cite{Fletcher70,Fletcher73}. Later on, Fletcher's
penalty function was modified and thoroughly investigated by many researchers
\cite{DiPilloGrippo89,DiPillo1994,MukaiPolak,GladPolak,BoggsTolle1980,Bertsekas_book,HanMangasarian_C1PenFunc,
DiPilloGrippo85,DiPilloGrippo86,Lucidi92,ContaldiDiPilloLucidi_93,FukudaSilva,AndreaniFukuda}. Here, we study a
modification of the continuously differentiable penalty function for nonlinear second-order cone programming problems
proposed by Fukuda, Silva and Fukushima in \cite{FukudaSilva}. However, it should be pointed out that the
results of this subsection can be easily extended to the case of any existing modification of Fletcher's penalty
function.

Let $X = A = \mathbb{R}^d$, and suppose that the set $M$ has the form
$$
  M = \Big\{ x \in \mathbb{R}^d \Bigm| g_i(x) \in Q_{l_i + 1}, \quad i \in I, \quad h(x) = 0, \Big\}
$$
where $g_i \colon X \to \mathbb{R}^{l_i + 1}$, $I = \{ 1, \ldots, r \}$, and $h \colon X \to \mathbb{R}^s$ are given
functions, and 
$$
  Q_{l_i + 1} = \big\{ y = (y^0, \overline{y}) \in \mathbb{R} \times \mathbb{R}^{l_i} \bigm| 
  y^0 \ge \| \overline{y} \| \big\}
$$
is the second order (Lorentz) cone of dimension $l_i + 1$ (here $\| \cdot \|$ is the Euclidean norm). In this case the
problem $(\mathcal{P})$ is a nonlinear second-order cone programming problem. 

Following the ideas of \cite{FukudaSilva}, let use introduce a continuously differentiable exact penalty function for
the problem under consideration. Suppose that the functions $f$, $g_i$, $i \in I$ and $h$ are twice continuously
differentiable. 
For any $\lambda = (\lambda_1, \ldots, \lambda_r) \in \mathbb{R}^{l_1 + 1} \times \ldots \times \mathbb{R}^{l_r + 1}$
and $\mu \in \mathbb{R}^s$ denote by 
$$
  L(x, \lambda, \mu) = f(x) + \sum_{i = 1}^r \langle \lambda_i, g_i(x) \rangle + \langle \mu, h(x) \rangle,
$$
the standard Lagrangian function for the nonlinear second-order cone programming problem. Here 
$\langle \cdot, \cdot \rangle$ is the inner product in $\mathbb{R}^k$.

For a chosen $x \in \mathbb{R}^n$ consider the following unconstrained minimization problem, which allows one to obtain
an estimate of Lagrange multipliers:
\begin{multline} \label{SOCP_LagrangeMultEstimate}
  \min_{\lambda, \mu} \big\| \nabla_x L(x, \lambda, \mu) \big\|^2 + 
  \zeta_1 \sum_{i = 1}^r \Big( \langle \lambda_i, g_i(x) \rangle^2 + 
  \| (\lambda_i)_0 \overline{g}_i(x) + (g_i)_0(x) \overline{\lambda}_i \|^2 \Big) \\
  + \frac{\zeta_2}{2} \left( \| h(x) \|^2 + \sum_{i = 1}^r \dist^2\big( g_i(x), Q_{l_i + 1} \big) \right) \cdot
  \big( \| \lambda \|^2 + \| \mu \|^2 \big),
\end{multline}
where $\zeta_1$ and $\zeta_2$ are some positive constants, 
$\lambda_i = ((\lambda_i)_0, \overline{\lambda}_i) \in \mathbb{R} \times \mathbb{R}^{l_i}$, and the same notation
is used for $g_i(x)$. Observe that if $(x^*, \lambda^*, \mu^*)$ is a KKT-point of the problem $(\mathcal{P})$, then
$(\lambda^*, \mu^*)$ is a globally optimal solution of problem \eqref{SOCP_LagrangeMultEstimate}
(see~\cite{FukudaSilva}). Moreover, it is easily seen that for any $x \in \mathbb{R}^d$ there exists a globally optimal
solution of this problem, which we denote by $(\lambda(x), \mu(x))$. In order to ensure that an optimal solution is
unique one has to utilize a proper constraint qualification. 

Recall that a feasible point $x$ is called \textit{nondegenerate} (\cite{BonnansShapiro}, Def.~4.70), if
$$
  \begin{bmatrix}
    J g_1(x) \\
    \vdots \\
    J g_r(x) \\
    J h(x)
  \end{bmatrix} \mathbb{R}^d +
  \begin{bmatrix}
    \lineal T_{Q_{l_1 + 1}} \big( g_1(x) \big) \\
    \vdots \\
    \lineal T_{Q_{l_r + 1}} \big( g_r(x) \big) \\
    \{ 0 \}
  \end{bmatrix} =
  \begin{bmatrix}
    \mathbb{R}^{l_1 + 1} \\
    \vdots \\
    \mathbb{R}^{l_r + 1} \\
    \mathbb{R}^s
  \end{bmatrix},
$$
where $J g_i(x)$ is the Jacobian of $g_i(x)$, ``lin'' stands for the lineality subspace of a convex cone, i.e. 
the largest linear space contained in this cone, and $T_{Q_{l_1 + 1}} \big( g_1(x) \big)$ is the contingent cone to
$Q_{l_i + 1}$ at the point $g_i(x)$. Let us note that the nondegeneracy condition can be expressed as a
``linear independence-type'' condition (see~\cite{FukudaSilva}, Lemma~3.1, and \cite{BonnansRamirez}, Proposition~19).
Furthermore, by \cite{BonnansShapiro}, Proposition~4.75, the nondegeneracy condition guarantees that if $x$ is a
locally optimal solution of the problem $(\mathcal{P})$, then there exists a \textit{unique} Lagrange multiplier at
$x$.

Suppose that every feasible point of the problem $(\mathcal{P})$ is nondegenerate. Then one can verify that a
globally optimal solution $(\lambda(x), \mu(x))$ of problem \eqref{SOCP_LagrangeMultEstimate} is unique for all 
$x \in \mathbb{R}^d$, and the functions $\lambda(\cdot)$ and $\mu(\cdot)$ are continuously differentiable
(\cite{FukudaSilva}, Proposition~3.3).

Now we can introduce a new continuously differentiable exact penalty function for nonlinear second-order cone
programming problems, which is a simple modification of the penalty function from \cite{FukudaSilva}. Namely, choose
$\alpha > 0$ and $\varkappa \ge 2$, and define
\begin{equation} \label{BarrierTerms_SOC}
  p(x) = \frac{a(x)}{1 + \sum_{i = 1}^r \| \lambda_i(x) \|^2}, \quad
  q(x) = \frac{b(x)}{1 + \| \mu(x) \|^2},
\end{equation}
where
$$
  a(x) = \alpha - \sum_{i = 1}^r \dist^{\varkappa}\big( g_i(x), Q_{l_i + 1} \big), \quad
  b(x) = \alpha - \| h(x) \|^2.
$$
Finally, denote $\Omega_{\alpha} = \{ x \in \mathbb{R}^d \mid a(x) > 0, \: b(x) > 0 \}$, and define
\begin{multline} \label{C1_PenaltyFunc}
  F(x, c) = f(x) \\
  + \frac{c}{2 p(x)} \sum_{i = 1}^r 
  \left[ \dist^2\Big( g_i(x) + \frac{p(x)}{c} \lambda_i(x), Q_{l_i + 1} \Big) - 
  \frac{p(x)^2}{c^2} \| \lambda_i(x) \|^2 \right] \\
  + \langle \mu(x), h(x) \rangle + \frac{c}{2 q(x)} \| h(x) \|^2,
\end{multline}
if $x \in \Omega_{\alpha}$, and $F(x, c) = + \infty$ otherwise. Let us point out that $F(x, c)$ is, in essence, a
straightforward modification of the Hestenes-Powell-Rockafellar augmented Lagrangian function to the case of nonlinear
second-order cone programming problems \cite{LiuZhang2007,LiuZhang2008,ZhouChen2015} with Lagrange multipliers $\lambda$
and $\mu$ replaced by their estimates $\lambda(x)$ and $\mu(x)$. One can easily verify that the function $F(\cdot, c)$
is l.s.c. on $\mathbb{R}^d$, and continuously differentiable on its effective domain (see~\cite{FukudaSilva}).

Let us obtain \textit{first} simple \textit{necessary and sufficient} conditions for the global exactness of
continuously differentiable penalty functions.

\begin{theorem}[Localization Principle for $C^1$ Penalty Functions] \label{Thrm_SOC_C1Penalty_LP}
Let the functions $f$, $g_i$, $i \in I$, and $h$ be twice continuously differentiable, and suppose that every feasible
point of the problem $(\mathcal{P})$ is nondegenerate. Then the continuously differentiable penalty function $F(x, c)$
is globally exact if and only if it is locally exact at every globally optimal solution of the problem $(\mathcal{P})$
and one of the two following assumptions is satisfied:
\begin{enumerate}
\item{the function $F(x, c)$ is non-degenerate;
}

\item{there exists $c_0 > 0$ such that the set $\{ x \in \mathbb{R} \mid F(x, c_0) < f^* \}$ is bounded.
}
\end{enumerate}
In particular, if the set $\{ x \in \mathbb{R}^d \mid f(x) < f^* + \gamma, \: a(x) > 0, \: b(x) > 0 \}$ is bounded for
some $\gamma > 0$, then $F(x, c)$ is globally exact if and only if it is locally exact at every globally optimal
solution of the problem $(\mathcal{P})$.
\end{theorem}

\begin{proof}
Our aim is to apply the localization principle in the parametric form to the separating function
\eqref{C1_PenaltyFunc}. To this end, define $G(\cdot) = (g_1(\cdot), \ldots, g_r(\cdot))$, 
$K = Q_{l_1 + 1} \times \ldots \times Q_{l_r + 1}$, and introduce the function
\begin{equation} \label{C1_Penalty_AuxiliaryFunc}
  \Phi(x, c) = \min_{y \in K - G(x)} \left( - p(x) \langle \lambda(x), y \rangle + \frac{c}{2} \| y \|^2 \right).
\end{equation}
Note that the minimum is attained at a unique point $y(x, c)$ due to the facts $K - G(x)$ is a closed convex cone, and
the function on the right-hand side of the above equality is strongly convex in $y$. Observe also that 
\begin{equation} \label{C1_PenaltyFunc_Represent}
  F(x, c) = f(x) + \frac{1}{p(x)} \Phi(x, c) + \langle \mu(x), h(x) \rangle + \frac{c}{2 q(x)} \| h(x) \|^2
\end{equation}
(see~\cite{ShapiroSun}, formulae (2.5) and (2.7)). Consequently, the function $F(x, c)$ is nondecreasing
in $c$.

From \eqref{C1_Penalty_AuxiliaryFunc} and \eqref{C1_PenaltyFunc_Represent} it follows that $F(x, c) \le f(x)$ for any
feasible point $x$ (in this case $y = 0 \in K - G(x)$). Let, now, $(x^*, \lambda^*, \mu^*)$ be a KKT-point of the
problem $(\mathcal{P})$. Then by \cite{FukudaSilva}, Proposition~3.3(c) one has $\lambda(x^*) = \lambda^*$ and 
$\mu(x^*) = \mu^*$, which, in particular, implies that $\lambda_i(x^*) \in (Q_{l_i + 1})^*$ and 
$\langle \lambda_i(x^*), g_i(x^*) \rangle = 0$, where $(Q_{l_i + 1})^*$ is the polar cone of $Q_{l_i + 1}$. Then
applying the standard first order necessary and sufficient conditions for a minimum of a convex function on a convex set
one can easily verify that the infimum in
$$
  \dist^2\Big( g_i(x^*) + \frac{p(x^*)}{c} \lambda_i(x^*), Q_{l_i + 1} \Big) = 
  \inf_{z \in Q_{l_i + 1}} \left\| g_i(x^*) + \frac{p(x^*)}{c} \lambda_i(x^*) - z \right\|^2
$$
is attained at the point $z = g_i(x^*)$. Therefore $F(x^*, c) = f(x^*)$ for all $c > 0$ (see~\eqref{C1_PenaltyFunc}).
In particular, if $x^*$ is a globally optimal solution of the problem $(\mathcal{P})$, then $F(x^*, c) \equiv f^*$.

Suppose that for some $c_0 > 0$ one has
\begin{equation} \label{C1_PenaltyFunc_IntersectImpliesExact}
  \Omega^* \cap \argmin_{x \in \mathbb{R}^d} F(x, c_0) \ne \emptyset.
\end{equation}
Then 
\begin{equation} \label{C1_PenaltyFunc_Min}
  \min_x F(x, c) = f^* = F(x^*, c) \quad \forall c \ge c_0 \quad \forall x^* \in \Omega^*
\end{equation}
due to the fact that $F(x, c)$ is nondecreasing in $c$. Thus, for all $c \ge c_0$ one has
$\Omega^* \subseteq \argmin_{x \in \mathbb{R}^d} F(x, c)$.

Let, now, $c > c_0$ and $x^* \in \argmin_x F(x, c)$ be arbitrary. Clearly, if $h(x^*) \ne 0$, then 
$F(x^*, c) > F(x^*, c_0)$, which is impossible. Therefore $h(x^*) = 0$. Let, now, $y(x^*, c)$ be such that
$$
  \Phi(x^*, c) = - p(x^*) \langle \lambda(x^*), y(x^*, c) \rangle + \frac{c}{2} \| y(x^*, c) \|^2
$$
(see~\eqref{C1_Penalty_AuxiliaryFunc}). From the definitions of $x^*$ and $c_0$ it follows that
\begin{multline*}
  f^* = F(x^*, c_0) = f(x^*) + \frac{1}{p(x^*)} \Phi(x^*, c_0) \\
  \le 
  f(x^*) + \frac{1}{p(x^*)} \left( - p(x^*) \langle \lambda(x^*), y \rangle + \frac{c_0}{2} \| y \|^2 \right)
\end{multline*}
for any $y \in K - G(x^*)$. Hence for any $y \in (K - G(x^*)) \setminus \{ 0 \}$ one has
$$
  f(x^*) + \frac{1}{p(x^*)} \left( - p(x^*) \langle \lambda(x^*), y \rangle + \frac{c}{2} \| y \|^2 \right) >
  f^*.
$$
Consequently, taking into account the first equality in \eqref{C1_PenaltyFunc_Min} and the definition of $x^*$ one
obtains that $y(x^*, c) = 0$ and $\Phi(x^*, c) = 0$, which yields that $0 \in K - G(x^*)$, i.e. $x^*$ is feasible, and
$F(x^*, c) = f(x^*) = f^*$. Therefore $x^* \in \Omega^*$. Thus, $\argmin_{x \in \mathbb{R}^d} F(x, c) = \Omega^*$ for
all $c > c_0$ or, in other words, the validity of condition \eqref{C1_PenaltyFunc_IntersectImpliesExact} implies
that the penalty function $F(x, c)$ is globally exact.

Let us now check that $F(x, c)$ is a penalty-type separating function. Then applying Theorems~\ref{Thrm_LocPrin_Param}
and \ref{Thrm_LocPrin_Param_SublevelSets} we arrive at the required result.

Indeed, let $\{ c_n \} \subset (0, + \infty)$ be an increasing unbounded sequence, $x_n \in \argmin_{x} F(x, c_n)$ for
all $n \in \mathbb{N}$, and $x^*$ be a cluster point of the sequence $\{ x_n \}$. As it was noted above, 
$F(y^*, c) = f^*$ for any globally optimal solution of the problem $(\mathcal{P})$. Therefore $F(x_n, c_n) \le f^*$ for
all $n \in \mathbb{N}$. On the other hand, minimizing the function $\omega(x, t) = - \| \mu(x) \| t + c t^2 / 2 q(x)$
with respect to $t$ one obtains that
\begin{equation} \label{C1PenFunc_GlobMin_LowerEstimate}
  F(x_n, c_n) \ge f(x_n) - \sum_{i = 1}^r \frac{p(x_n)}{2 c_n} \| \lambda_i(x_n) \|^2 
  - \frac{q(x_n)}{2 c_n} \| \mu(x_n) \|^2 \ge f(x_n) - \frac{\alpha}{c_n}.
\end{equation}
Hence passing to the limit as $n \to + \infty$ one obtains that $f(x^*) \le f^*$. Therefore it remains to show that
$x^*$ is a feasible point of the problem $(\mathcal{P})$.

Arguing by reductio ad absurdum, suppose that $x^*$ is not feasible. Let, at first, $h(x^*) \ne 0$. Then there
exist $\varepsilon > 0$ and a subsequence $\{ x_{n_k} \}$ such that $\| h(x_{n_k}) \| \ge \varepsilon$ for
all $k \in \mathbb{N}$. Note that since $\{ x_n \}$ is a convergent sequence and the function $\mu(\cdot)$ is
continuous, there exists $\mu_0 > 0$ such that $\| \mu(x_n) \| \le \mu_0$ for all $n \in \mathbb{N}$. Furthermore,
it is obvious that $\| h(x_n) \|^2 < \alpha$ for all $n \in \mathbb{N}$. Consequently, one has
$$
  F(x_{n_k}, c_{n_k}) \ge f(x_n) - \frac{\alpha}{2 c_{n_k}} - \mu_0 \alpha + 
  \frac{c_{n_k} \varepsilon^2}{2(\alpha - \varepsilon^2)}.
$$
(clearly, one can suppose that $\varepsilon^2 < \alpha$). Therefore $\limsup_{n \to \infty} F(x_n, c_n) = + \infty$,
which is impossible. Thus, $h(x^*) = 0$.

Suppose, now, that $g_i(x^*) \notin Q_{l_i + 1}$ for some $i \in I$. Then there exist $\varepsilon > 0$ and a
subsequence $\{ x_{n_k} \}$ such that $\dist( g_i(x_{n_k}), Q_{l_i + 1}) \ge \varepsilon$ for all $k \in \mathbb{N}$.
Note that $p(x) \| \lambda_i(x) \| / c < \alpha / c$. Consequently, one has 
$\dist( g_i(x_{n_k}) + p(x_{n_k}) \lambda_i(x_{n_k}) / c_{n_k}, Q_{l_i + 1}) \ge \varepsilon / 2$ for any
sufficiently large $k \ge n$. Therefore
$$
  F(x_{n_k}, c_{n_k}) \ge f(x_{n_k}) + 
  \frac{c_{n_k} \varepsilon^2}{2 (\alpha - \varepsilon^{\varkappa})}  - \frac{\alpha}{c_{n_k}}
$$
for any $k$ large enough (obviously, we can assume that $\varepsilon^{\varkappa} < \alpha$). Passing to the limit as $k
\to \infty$ one obtains that $\limsup_{n \to \infty} F(x_n, c_n) = + \infty$, which is impossible. Thus, $x^*$ is a
feasible point of the problem $(\mathcal{P})$.

Finally, note that from \eqref{C1PenFunc_GlobMin_LowerEstimate} it follows that
$$
  \big\{ x \in \mathbb{R}^d \bigm| F(x, c) < f^* \big\} \subseteq 
  \big\{ x \in \mathbb{R}^d \bigm| f(x) < f^* + \gamma, \: a(x) > 0, \: b(x) > 0 \big\}
$$
for all $c > \alpha / \gamma$.
\end{proof}

\begin{remark} \label{Remark_C1PenFunc}
{(i)~Let us note that the local exactness of penalty function \eqref{C1_PenaltyFunc} at a globally optimal solution of
the problem $(\mathcal{P})$ can be easily established with the use of second sufficient optimality conditions 
(see~\cite{FukudaSilva}, Theorem~5.7).
}

\noindent{(ii)~Note that from the proof of the theorem above it follows that $F(x^*, c) = f(x^*)$ for any KKT-point
$(x^*, \lambda^*, \mu^*)$ of the problem $(\mathcal{P})$.
}
\end{remark}

Following the underlying idea of the localization principle and utilizing some specific properties of continuously
differentiable penalty function \eqref{C1_PenaltyFunc} we can obtain stronger necessary and sufficient conditions for
the global exactness of this function than the ones in the theorem above. These conditions does not rely on the local
exactness property and, furthermore, strengthen existing \textit{sufficient} conditions for global exactness of
continuously differentiable exact penalty functions for nonlinear second-order cone programming problems
(\cite{FukudaSilva}, Proposition~4.9). However, it should be emphasized that these conditions heavily rely on 
the particular structure of the penalty function under consideration.

\begin{theorem} \label{Thrm_SOC_C1Penalty_Mod}
Let the functions $f$, $g_i$, $i \in I$, and $h$ be twice continuously differentiable, and suppose that every feasible
point of the problem $(\mathcal{P})$ is nondegenerate. Then the continuously differentiable penalty function $F(x, c)$
is globally exact if and only if there exists $c_0 > 0$ such that 
the set $\{ x \in \mathbb{R}^d \mid F(x, c_0) < f^* \}$ is bounded. In particular, it is exact, if there exists 
$\gamma > 0$ such that the set $\{ x \in \mathbb{R}^d \mid f(x) < f^* + \gamma, \: a(x) > 0, \: b(x) > 0 \}$ is bounded.
\end{theorem}

\begin{proof}
Denote $S(c) = \{ x \in \mathbb{R}^d \mid F(x, c) < f^* \}$. If $F(x, c)$ is globally exact, then, as it is easy to
check, one has $S(c) = \emptyset$. Therefore it remains to prove the ``if'' part of the theorem. 

If $S(c) = \emptyset$ for some $c > 0$, then  $F(x, c) \ge f^*$ for all $x \in \mathbb{R}^d$, and arguing in the same
way as at the beginning of the proof of Theorem~\ref{Thrm_SOC_C1Penalty_LP} one can easily obtain the desired result.
Thus, one can suppose that $S(c) \ne \emptyset$ for all $c > 0$. 

Choose an increasing unbounded sequence $\{ c_n \} \subset [c_0, + \infty)$. Taking into account the facts that 
$F(\cdot, c)$ is l.s.c. and nondecreasing in $c$, and applying the boundedness of the set $S(c_0)$ one obtains that for
any $n \in \mathbb{N}$ the function $F(\cdot, c_n)$ attains a global minimum at a point $x_n \in S(c_n) \subseteq
S(c_0)$. Applying the boundedness of the set $S(c_0)$ once again one obtains that there exists a cluster point $x^*$ of
the sequence $\{ x_n \}$. Replacing, if necessary, the sequence $\{ x_n \}$ with its subsequence we can suppose that
$x_n$ converges to $x^*$. As it was shown in Theorem~\ref{Thrm_SOC_C1Penalty_LP}, $F(x, c)$ is a penalty-type separating
function. Therefore $x^*$ is a globally optimal solution of the problem $(\mathcal{P})$, and $F(x^*, c) = f^*$ for all
$c > 0$.

From the fact that $x_n$ is a point of global minimum of $F(\cdot, c_n)$ it follows that $\nabla_x F(x_n, c_n) = 0$.
Then applying a direct modification of \cite{FukudaSilva}, Proposition~4.3 to the case of penalty function
\eqref{C1_PenaltyFunc} one obtains that for any $x_n$ in a sufficiently small neighbourhood of $x^*$ (i.e.
for any sufficiently large $n \in \mathbb{N}$) the triplet $(x_n, \lambda(x_n), \mu(x_n))$ is a KKT-point of the problem
$(\mathcal{P})$. Hence taking into account Remark~\ref{Remark_C1PenFunc} one gets that $F(x_n, c_n) = f(x_n) \ge f^*$
for any sufficiently large $n \in \mathbb{N}$, which contradicts our assumption that $S(c) \ne \emptyset$ for all $c >
0$ due to the definition of $x_n$. Thus, the penalty function $F(x, c)$ is globally exact.
\end{proof}

Let us note that the results of this subsection can be easily extended to the case of nonlinear semidefinite
programming problems (cf.~\cite{Dolgopolik_GSP}, Sections~8.3 and 8.4). Namely, suppose that $A = \mathbb{R}^d$, and let
$$
  M = \Big\{ x \in \mathbb{R}^d \Bigm| G(x) \preceq 0, \: h(x) = 0 \Big\},
$$
where $G \colon X \to \mathbb{S}^l$ and $h \colon X \to \mathbb{R}^s$ are given functions, $\mathbb{S}^l$ is the set of
all $l \times l$ real symmetric matrices, and the relation $G(x) \preceq 0$ means that the matrix $G(x)$ is negative
semidefinite. We suppose that the space $\mathbb{S}^l$ is equipped with the Frobenius norm 
$\| A \|_F = \sqrt{\trace(A^2)}$. In this case the problem $(\mathcal{P})$ is a nonlinear semidefinite programming
problem. 

Suppose that the functions $f$, $G$ and $h$ are twice continuously differentiable. For any $\lambda \in \mathbb{S}^l$
and $\mu \in \mathbb{R}^s$ denote by
$$
  L(x, \lambda, \mu) = f(x) + \trace( \lambda G(x) ) + \langle \mu, h(x) \rangle,
$$
the standard Lagrangian function for the nonlinear semidefinite programming problem. For a chosen $x \in \mathbb{R}^n$
consider the following unconstrained minimization problem, which allows one to compute an estimate of Lagrange
multipliers:
\begin{multline} \label{SemiDef_LagrangeMultEstimate}
  \min_{\lambda, \mu} \big\| \nabla_x L(x, \lambda, \mu) \big\|^2 + 
  \zeta_1 \trace( \lambda^2 G(x)^2 ) \\
  + \frac{\zeta_2}{2} \left( \| h(x) \|^2 + \sum_{i = 1}^r \dist^2\big( G(x), S^l_{-} \big) \right) \cdot
  \big( \| \lambda \|_F^2 + \| \mu \|^2 \big),
\end{multline}
where $\zeta_1$ and $\zeta_2$ are some positive constants, and $S^l_{-}$ is the cone of $l \times l$ real negative
semidefinite matrices. One can verify (cf.~\cite{Dolgopolik_GSP}, Lemma~4) that for any $x \in \mathbb{R}^d$ there
exists a unique globally optimal solution $(\lambda(x), \mu(x))$ of this problem, provided every feasible point of the
problem $(\mathcal{P})$ is non-degenerate, i.e. provided for any feasible $x$ one has
$$
  \begin{bmatrix}
    D G(x_*) \\
    J h(x_*)
  \end{bmatrix} \mathbb{R}^d +
  \begin{bmatrix}
    \lineal T_{\mathbb{S}^l_-} \big( G(x_*) \big) \\
    \{ 0 \}
  \end{bmatrix} =
  \begin{bmatrix}
    \mathbb{S}^l \\
    \mathbb{R}^s
  \end{bmatrix}.
$$
(see \cite{BonnansShapiro}, Def.~4.70). Let us note that, as in the case of second order cone programming problems, the
above nondegeneracy condition can be rewritten as a ``linear independence-type'' condition (see \cite{BonnansShapiro},
Proposition~5.71).

Now we can introduce \textit{first} continuously differentiable \textit{exact} penalty function for nonlinear
semidefinite programming problems. Namely, choose $\alpha > 0$ and $\varkappa \ge 1$, and define
$$
  p(x) = \frac{a(x)}{1 + \trace(\lambda(x)^2)}, \quad
  q(x) = \frac{b(x)}{1 + \| \mu(x) \|^2},
$$
where
$$
  a(x) = \alpha - \trace\big( [ G(x) ]_+^2 \big)^{\varkappa}, \quad
  b(x) = \alpha - \| h(x) \|^2,
$$
and $[\cdot]_+$ denotes the projection of a matrix onto the cone of $l \times l$ positive semidefinite matrices. 
Denote $\Omega_{\alpha} = \{ x \in \mathbb{R}^d \mid a(x) > 0, \: b(x) > 0 \}$, and define
\begin{multline} \label{C1_PenaltyFunc_SemiDef}
  F(x, c) = f(x) 
  + \frac{1}{2c p(x)} \Big( \trace\big( [c G(x) + p(x) \lambda(x)]_+^2 \big) - 
  p(x)^2 \trace(\lambda(x)^2) \Big) \\
  + \langle \mu(x), h(x) \rangle + \frac{c}{2 q(x)} \| h(x) \|^2,
\end{multline}
if $x \in \Omega_{\alpha}$, and $F(x, c) = + \infty$ otherwise. Let us point out that $F(x, c)$ is, in essence, 
a direct modification of the Hestenes-Powell-Rockafellar augmented Lagrangian function to the case of nonlinear
semidefinite programming problems
\cite{SunZhangWu2006,SunSunZhang2008,ZhaoSunToh2010,Sun2011,WenGoldfarbYin2010,LuoWuChen2012,WuLuoDingChen2013,
WuLuoYang2014,YamashitaYabe2015} with Lagrange multipliers $\lambda$ and $\mu$ replaced by their estimates $\lambda(x)$
and $\mu(x)$. One can verify that the function $F(\cdot, c)$ is l.s.c. on $\mathbb{R}^d$, and continuously
differentiable on its effective domain. Furthermore it is possible to extend Theorems~\ref{Thrm_SOC_C1Penalty_LP} and
\ref{Thrm_SOC_C1Penalty_Mod} to the case of continuously differentiable penalty function \eqref{C1_PenaltyFunc_SemiDef},
thus obtaining \textit{first} necessary and sufficient conditions for the global exactness of $C^1$ penalty functions
for nonlinear semidefinite programming problems. However, we do not present the proofs of these results here, and leave
them to the interested reader.

\subsection{Example IV: Rockafellar-Wets' Augmented Lagrangian Function}
\label{Example_ALRW}

The separating functions studied in the previous examples do not depend on any additional parameters apart from
the penalty parameter $c$. This fact does not allow one to fully understand the concept of parametric exactness. In
order to illuminate the main features of parametric exactness, in this example we consider a separating function that
depends on additional parameters, namely Lagrange multipliers. Below, we apply the general theory of parametrically
exact separating functions to the augmented Lagrangian function introduced by Rockafellar and Wets in \cite{RockWets}
(see also
\cite{ShapiroSun,HuangYang2003,HuangYang2005,ZhouZhouYang2014,Dolgopolik_ALRW,RuckmannShapiro,KanSong,
KanSong2,BurachikYangZhou2017}). 

Let $P$ be a topological vector space of parameters. Recall that a function 
$\Phi \colon X \times P \to \mathbb{R} \cup \{ + \infty \} \cup \{ - \infty \}$ is called \textit{a dualizing
parameterization function} for $f$ iff $\Phi(x, 0) = f(x)$ for any feasible point of the problem $(\mathcal{P})$.
A function $\sigma \colon P \to [0, + \infty]$ such that $\sigma(0) = 0$ and $\sigma(p) > 0$ for all $p \ne 0$ is called
\textit{an augmenting function}. Let, finally, $\Lambda$ be a vector space of \textit{multipliers}, and let the pair
$(\Lambda, P)$ be equipped with a bilinear coupling function 
$\langle \cdot, \cdot \rangle \colon \Lambda \times P \to \mathbb{R}$.

Following the ideas of Rockafellar and Wets \cite{RockWets}, define the augmented Lagrangian function
\begin{equation} \label{ALRW_Def}
  \mathscr{L}(x, \lambda, c) = 
  \inf_{p \in P} \Big( \Phi(x, p) - \langle \lambda, p \rangle + c \sigma(p) \Big),
\end{equation}
We suppose that $\mathscr{L}(x, \lambda, c) > - \infty$ for all $x \in X$, $\lambda \in \Lambda$ and $c > 0$. Let us
obtain simple necessary and sufficient conditions for the strict global parametric exactness of the augmented Lagrangian
function $\mathscr{L}(x, \lambda, c)$ with the use of the localization principle. These conditions were first obtained
by the author in \cite{Dolgopolik_ALRW}.

\begin{remark}
It is worth mentioning that in the context of the theory of augmented Lagrangian functions, a vector 
$\lambda^* \in \Lambda$ is a strictly exact tuning parameter of the function $\mathscr{L}(x, \lambda, c)$ iff
$\lambda^*$ \textit{supports an exact penalty representation of the problem} $(\mathcal{P})$ (see~\cite{RockWets},
Definition~11.60). Furthermore, if the infimum in \eqref{ALRW_Def} is attained for all $x$, $\lambda$ and $c$, then the
strict global parametric exactness of the augmented Lagrangian function $\mathscr{L}(x, \lambda, c)$ is equivalent to
the existence of \textit{an augmented Lagrange multiplier} (see~\cite{RockWets}, Theorem~11.61, and
\cite{Dolgopolik_ALRW}, Proposition~4 and Corollary~1). Furthermore, in this case $\lambda^*$ is a strictly exact tuning
parameter iff it is an augmented Lagrange multiplier.
\end{remark}

\begin{remark}
Clearly, the definitions of strict global parametric exactness and global parametric exactness do not coincide in the
case of the augmented Lagrangian function $\mathscr{L}(x, \lambda, c)$ defined above. For example, let $P$ be a normed
space, $\Lambda$ be the topological dual of $P$, $\sigma(p) = \| p \|^2$, and
$$
  \Phi(x, p) = \begin{cases}
    f(x) + \max\{ - 1, - \| p \| \}, & \text{if } x \in \Omega, \\
    + \infty, & \text{if } x \notin \Omega.
  \end{cases},
$$
Then, as it easy to check, $\mathscr{L}(x, \lambda, c)$ is globally parametrically exact with the exact tuning
parameter $\lambda^* = 0$ and $c^*(\lambda^*) = 0$, but it is not strictly globally parametrically exact, since 
$\mathscr{L}(x, \lambda, c) < f(x)$ for all $x \in \Omega$, $\lambda \in \Lambda$ and $c > 0$. When one compares 
strict global parametric exactness and global parametric exactness in the case of the augmented Lagrangian
$\mathscr{L}(x, \lambda, c)$, it appears that the strict global parametric exactness is more natural
in this case. Apart from the fact that there exist many connections of the strict global parametric exactness with
existing results on augmented Lagrangian function $\mathscr{L}(x, \lambda, c)$ pointed out above, the application of the
localization principle leads to simpler results in the case of the strict global parametric exactness. In particular, it
is rather difficult to verify that the validity of the condition $\Omega^* \cap \argmin_{x \in A} \mathscr{L}(x,
\lambda^*, c) \ne \emptyset$ implies the global parametric exactness, while the condition
\begin{equation} \label{RW_ALF_EquivExactnessCond}
  \Omega^* \cap \argmin_{x \in A} \mathscr{L}(x, \lambda^*, c) \ne \emptyset, \quad 
  \min_{x \in A} \mathscr{L}(x, \lambda^*, c) = f^*
\end{equation}
is equivalent to the strict global parametric exactness of $\mathscr{L}(x, \lambda, c)$ under some natural assumptions 
(see Theorem~\ref{Thrm_LocPrin_ALRW} below).
\end{remark}

Recall that the augmenting function $\sigma$ is said \textit{to have a valley at zero} iff for any neighbourhood 
$U \subset P$ of zero there exists $\delta > 0$ such that $\sigma(p) \ge \delta$ for any $p \in P \setminus U$. The
assumption that the augmenting function $\sigma$ has a valley at zero is widely used in the literature on augmented
Lagrangian functions (see, e.g.,~\cite{BurachikRubinov,ZhouYang2009,ZhouYang2012,ZhouZhouYang2014}).

\begin{theorem}[Localization Principle for Augmented Lagrangian Functions] \label{Thrm_LocPrin_ALRW}
Suppose that the following assumptions are valid:
\begin{enumerate}
\item{$A$ and $\Omega$ are closed;
\label{Assumpt_ALRW_Closedness}
}

\item{$f$ and $\mathscr{L}(\cdot, \lambda, c)$ for all $\lambda \in \Lambda$ and $c > 0$ are l.s.c. on $A$;}

\item{$\Phi$ is l.s.c. on $A \times \{ 0 \}$;}

\item{$\sigma$ has a valley at zero;
\label{Assumpt_ALRW_ValleyAtZero}
}

\item{there exists $r > 0$ such that for any $c \ge r$, $x \in A$ and $\lambda \in \Lambda$ one has
$$
  \argmin_{p \in P} \Big( \Phi(x, p) - \langle \lambda, p \rangle + c \sigma(p) \Big) \ne \emptyset
$$
i.e. the infimum in \eqref{ALRW_Def} is attained.
\label{Assumpt_ALRW_InfAttainment}
}
\end{enumerate}
Then the augmented Lagrangian function $\mathscr{L}(x, \lambda, c)$ is strictly globally parametrically exact if and
only if there exist $\lambda^*$ and $c_0 > 0$ such that $\mathscr{L}(x, \lambda, c)$ is locally parametrically exact at
every globally optimal solution of the problem $(\mathcal{P})$ with the exact tuning parameter $\lambda^*$,
$$
  \mathscr{L}(x^*, \lambda^*, c) = f^* \quad \forall x^* \in \Omega^* \quad \forall c \ge c_0,
$$
and one of the following two conditions is valid:
\begin{enumerate}
\item{the function $\mathscr{L}(x, \lambda, c)$ is non-degenerate for $\lambda = \lambda^*$;
\label{Assumpt_ALRW_1}
}

\item{the set $\{ x \in A \Bigm| \mathscr{L}(x, \lambda^*, c_0) < f^* \Big\}$ is bounded.
\label{Assumpt_ALRW_2}
}
\end{enumerate}
\end{theorem}

\begin{proof}
The fact that the validity of \eqref{RW_ALF_EquivExactnessCond} is equivalent to strict global parametric exactness of
the function $\mathscr{L}(x, \lambda, c)$ follows directly from \cite{Dolgopolik_ALRW}, Proposition~4
and Corollary~1. Furthermore, by \cite{Dolgopolik_ALRW}, Proposition~8, the function $\mathscr{L}(x, \lambda, c)$ is 
a penalty-type separating function for any $\lambda \in \Lambda$. Then applying
Theorems~\ref{Thrm_LocPrin_Param_Strict} and \ref{Thrm_LocPrin_Param_Strict_Sublevel} one obtains the desired result.
\end{proof}

\begin{remark}
Note that under the assumptions of the theorem the multipliers $\lambda^*$ is a strictly exact tuning parameter
if and only if it is an augmented Lagrangre multiplier \cite{Dolgopolik_ALRW}. Thus, the theorem above, in essence,
contains necessary and sufficient conditions for the existence of augmented Lagrange multipliers for the problem
$(\mathcal{P})$. See~\cite{Dolgopolik_ALRW} for applications of this theorem to some particular optimization
problems.
\end{remark}

Note that from the localization principle it follows that for the strict global parametric exactness of the augmented
Lagrangian $\mathscr{L}(x, \lambda, c)$ it is necessary that there exists a tuning parameter $\lambda^* \in \Lambda$
such that $\lambda^*$ is a locally exact tuning parameter at \textit{every} globally optimal solution of the problem
$(\mathcal{P})$. One can give a simple interpretation of this condition in the case when $\mathscr{L}(x, \lambda, c)$ is
a proximal Lagrangian. Namely, let $\mathscr{L}(x, \lambda, c)$ be the proximal Lagrangian (see~\cite{RockWets},
Example~11.57), and suppose that it is strictly globally parametrically exact with a strictly exact tuning parameter 
$\lambda^* \in \Lambda$. By the definition of strict global exactness, any globally optimal solution $x^*$ of the
problem $(\mathcal{P})$ is a global minimizer of the function $L(\cdot, \lambda^*, c)$ for all sufficiently large $c$.
Then applying the first order necessary optimality condition to the function $\mathscr{L}(\cdot, \lambda, c)$ one can
easily verify that under natural assumptions the pair $(x^*, \lambda^*)$ is a KKT-point of the problem $(\mathcal{P})$
for any $x^* \in \Omega^*$ (see~\cite{ShapiroSun}, Proposition~3.1). Consequently, one gets that for the strict global
parametric exactness of the augmented Lagrangian function $\mathscr{L}(x, \lambda, c)$ it is \textit{necessary} that
there exists a Lagrange multiplier $\lambda^*$ such that the pair $(x^*, \lambda^*)$ is a KKT-point of the problem
$(\mathcal{P})$ for \textit{any} globally optimal solution $x^*$ of this problem. In particular, if there exist two
globally optimal solutions of the problem $(\mathcal{P})$ with disjoint sets of Lagrange multipliers, then the proximal
Lagrangian cannot be strictly globally parametrically exact.

For the sake of completeness, let us mention that in the case of augmented Lagrangian functions, sufficient conditions
for the local exactness are typically derived with the use of sufficient optimality conditions. In particular, the
validity of the second order sufficient optimality conditions at a given globally optimal solution $x^*$
guarantees that the proximal Lagrangian is locally parametrically exact at $x^*$ with the corresponding Lagrange
multiplier being a locally exact tuning parameter (see, e.g., \cite{Bazaraa}, Theorem~9.3.3; \cite{SunLi}, Theorem~2.1;
\cite{LiuTangYang2009}, Theorem~2; \cite{WangZhouXu2009}, Theorem~2.3; \cite{LuoMastroeniWu2010}, Theorems~3.1 and 3.2;
\cite{ZhouXiuWang}, Theorem~2.8; \cite{ZhouZhouYang2014}, Proposition~3.1; \cite{ZhouChen2015}, Theorem~2.3;
\cite{WuLuoYang2014}, Theorem~3, etc.).

\section{Conclusions}

In this paper we developed a general theory of global parametric exactness of separating function for finite-dimensional
constrained optimization problems. This theory allows one to reduce a constrained optimization problem to an
unconstrained one, provided an exact tuning parameter is known. With the use of the general results obtained in this
article we recovered existing results on the global exactness of linear penalty functions and Rockafellar-Wets'
augmented Lagrangian function. We also obtained new simple necessary and sufficient conditions for the global exactness
of nonlinear and continuously differentiable penalty functions.

\bibliographystyle{abbrv}  
\bibliography{ESF_I_bibl}

\end{document}